\documentclass[12pt]{article}
\usepackage{amsmath,amssymb}
\usepackage{amsthm}
\usepackage{graphicx,psfrag,epsf}
\usepackage{enumerate}
\usepackage{natbib}
\usepackage{multirow}
\usepackage{float}

\usepackage[
singlelinecheck=false 
]{caption}

\usepackage{bm}

\usepackage{url} 

\usepackage{bbm}

\usepackage{amsfonts}

\newcommand{\blind}{0}

\theoremstyle{remark}
\newtheorem{theorem}{Theorem}
\newtheorem{corollary}{Corollary}[theorem]
\newtheorem{lemma}{Lemma}

\newtheorem{example}{Example}

\DeclareMathOperator{\E}{\mathbb{E}}

\begin{document}

\def\spacingset#1{\renewcommand{\baselinestretch}%
{#1}\small\normalsize} \spacingset{1}


\if0\blind
{
  \title{\bf Sharpening Jensen's Inequality}
 \author{J. G. Liao and Arthur Berg 
    \hspace{.2cm}\\
    Division of Biostatistics and Bioinformatics\\ Penn State University College of Medicine}
  \maketitle
} \fi

\if1\blind
{
  \bigskip
  \bigskip
  \bigskip
  \begin{center}
    {\LARGE\bf Title}
\end{center}
  \medskip
} \fi

\bigskip
\begin{abstract}
 This paper proposes a new sharpened version of the Jensen's inequality. The proposed new bound is simple and insightful, is broadly applicable by imposing minimum assumptions, and provides fairly accurate result in spite of its simple form. Applications to the moment generating function, power mean inequalities, and Rao-Blackwell estimation are presented. This presentation can be incorporated in any calculus-based statistical course.
\end{abstract}

\noindent%
\vfill

\newpage
\spacingset{1.45} 

\section{Introduction}
Jensen's inequality is a fundamental inequality in mathematics and it underlies many important statistical proofs and concepts. Some standard applications include derivation of the arithmetic-geometric mean inequality, non-negativity of Kullback and Leibler divergence, and the convergence property of the expectation-maximization algorithm \citep{Dempster:1977aa}. Jensen's inequality is covered in all major statistical textbooks such as \citet[Section 4.7]{Casella:2002aa} and \citet[Section 4.2]{Wasserman:2013aa} as a basic mathematical tool for statistics.

Let $X$ be a random variable with finite expectation and let $\varphi(x)$ be a convex function, then Jensen's inequality \citep{Jensen:1906aa} establishes 
\begin{equation}
\label{eq:Jensen}
\mathbb{E}\left[\varphi \left(X\right)\right] - \varphi\left(\mathbb{E}\left[X\right]\right)\ge 0.
\end{equation}
This inequality, however, is not sharp unless $\text{var}(X)=0$ or $\varphi(x)$ is a linear function of $x$.  Therefore, there is substantial room for advancement. This paper proposes a new sharper bound for the Jensen gap $\E[\varphi(X)]-\varphi\left(\E[X]\right)$. Some other improvements of Jensen's inequality have been developed recently; see for example \cite{Walker:2014aa}, \cite{Abramovich:2016aa}; \cite{Horvath:2014aa} and references cited therein. Our proposed bound, however, has the following advantages. First, it has a simple, easy to use, and insightful form in terms of the second derivative $\varphi''(x)$ and $\text{var}(X)$.  At the same time, it gives fairly accurate results in the several examples below. Many previously published improvements, however, are much more complicated in form, much more involved to use, and can even be more difficult to compute than $E[\varphi(X)]$ itself as discussed in \cite{Walker:2014aa}. Second, our method requires only the existence of $\varphi''(x)$ and is therefore broadly applicable. In contrast, some other methods require $\varphi(x)$ to admit a power series representation with positive coefficients \citep{Abramovich:2016aa,Dragomir:2014aa,Walker:2014aa} or require $\varphi(x)$ to be super-quadratic \citep{Abramovich:2014aa}. Third, we provide both a lower bound and an upper bound in a single formula. 

We have incorporated the materials in this paper in our classroom teaching. With only slightly increased technical level and lecture time, we are able to present a much sharper version of the Jensen's inequality that significantly enhances students' understanding of the underlying concepts. 

\section{Main result}
\begin{theorem}
Let $X$ be a one-dimensional random variable with mean $\mu$, and $P(X\in(a,b))=1$, where $-\infty\le a<b\le \infty$.  Let $\varphi(x)$ is a twice differentiable function on $(a,b)$, and define function
\[
h(x;\nu)\triangleq\frac{\varphi \left(x\right)-\varphi \left(\nu \right)}{\left(x-\nu
\right)^2}-\frac{\varphi '\left(\nu \right)}{x-\nu}.
\]
Then
\begin{equation}
\label{eq:bounds}
\underset{x{\in}(a,b)}{\inf}\{h(x;\mu)\}\text{var}(X) \leq E\left[\varphi \left(X\right)\right]-\varphi\left(E[X]\right)\le \sup_{x{\in}(a,b)}\{h(x;\mu)\} \text{var}(X).
\end{equation}
\end{theorem}

\begin{proof} Let $F(x)$ be the cumulative distribution function of $X$. Applying Taylor's theorem to $\varphi(x)$ about $\mu$ with a mean-value form of the remainder gives
\begin{equation*}
\varphi(x)=\varphi(\mu)+\varphi'(\mu)(x-\mu)+\frac{\varphi''(g(x))}{2}(x-\mu)^2,
\end{equation*}
where $g(x)$ is between $x$ and $\mu$.  Explicitly solving for $\varphi''(g(x))/2$ gives $\varphi''(g(x))/2=h(x;\mu)$ as defined above.  Therefore
\[
\begin{split}
\mathbb{E}\left[\varphi \left(X\right)\right] - \varphi\left(\mathbb{E}\left[X\right]\right)
&=\int_a^b \left\{\varphi(x)-\varphi(\mu)\right\}\,dF(x)\\
&=\int_a^b \left\{\varphi'(\mu)(x-\mu)+h(x;\mu)(x-\mu)^2\right\}\,dF(x)\\
&=\int_a^b h(x;\mu)(x-\mu)^2\,dF(x),\\
\end{split}
\]
and the result follows because $\inf_{x\in(a,b)} h(x;\mu)\le h(x;\mu)\le \sup_{x\in(a,b)} h(x;\mu)$.  
\end{proof}

Theorem 1 also holds when $\inf h(x;\mu)$ is replaced by $\inf \varphi''(x)/2$ and $\sup h(x;\mu)$ replaced by $\sup \varphi''(x)/2$ since
\[
\inf \frac{\varphi''(x)}{2}\le \inf h(x;\mu) \quad\text{and}\quad \sup \frac{\varphi''(x)}{2}\ge \sup h(x;\mu).
\] 
These less tight bounds are implied in the economics working paper \cite{Becker:2012aa}.  Our lower and upper bounds have the general form $J\cdot\text{var}(X)$, where $J$ depends on $\varphi$.  Similar forms of bounds are presented in \cite{Abramovich:2016aa,Dragomir:2014aa,Walker:2014aa}, but our $J$ in Theorem 1 is much simpler and applies to a wider class of $\varphi$.

Inequality \eqref{eq:bounds} implies Jensen's inequality when $\varphi''(x)\ge 0$. Note also that Jensen's inequality is sharp when $\varphi(x)$ is linear, whereas inequality \eqref{eq:bounds} is sharp when $\varphi(x)$ is a quadratic function of $x$.  

In some applications the moments of $X$ present in \eqref{eq:bounds} are unknown, although a random sample $x_1,\ldots,x_n$ from the underlying distribution $F$ is available.  A version of Theorem 1 suitable for this situation is given in the following corollary.

\begin{corollary}  Let $x_1,\ldots,x_n$ be any $n$ datapoints in $(-\infty,\infty)$, and let 
\label{cor}
\[
\bar{x}=\frac{1}{n}\sum_{i=1}^n x_i,\quad \overline{\varphi_x}=\frac{1}{n}\sum_{i=1}^n\varphi(x_i),\quad S^2=\frac{1}{n}\sum_{i=1}^n(x_i-\bar{x})^2.
\]
Then
\[
\inf_{x\in[a,b]}h(x;\bar{x})S^2\le \overline{\varphi_x}-\varphi(\bar{x})\le\sup_{x\in[a,b]}h(x;\bar{x})S^2,
\]
where $a=\min\{x_1,\ldots,x_n\}$ and $b=\max\{x_1,\ldots,x_n\}$.  
\end{corollary}
\begin{proof}
Consider the discrete random variable $X$ with probability distribution $P(X=x_i)=1/n$, $i=1,\ldots,n$.  We have $E[X]=\bar{x}$, $E[\varphi(X)]=\overline{\varphi_x}$, and $\text{var}(X)=S^2$.  Then the corollary follows from application of Theorem 1.  
\end{proof}

\begin{lemma}
If $\varphi '\left(x\right)$ is convex, then $h(x;\mu)$ is monotonically increasing in $x$, and if $\varphi '\left(x\right)$ is concave, then $h(x;\mu)$ is monotonically decreasing in $x$.  
\end{lemma}

\begin{proof} We prove that $h'(x;\mu)\ge0$ when $\varphi'(x)$ is convex. The analogous result for concave $\varphi'(x)$ follows similarly.  Note that 
\begin{equation*}
\frac{\mathit{dh}(x;\mu)}{\mathit{dx}}=\frac{\frac{\varphi '\left(x\right)+\varphi '\left(\mu \right)} 2-\frac{\varphi
\left(x\right)-\varphi \left(\mu \right)}{x-\mu }}{\frac{1}{2}\left(x-\mu \right)^2,}
\end{equation*}
so it suffices to prove 
\begin{equation*}
\frac{\varphi '\left(x\right)+\varphi '\left(\mu \right)} 2\ge \frac{\varphi
\left(x\right)-\varphi \left(\mu \right)}{x-\mu }.
\end{equation*}
Without loss of generality we assume $x>\mu$.  Convexity of $\varphi'(x)$ gives 
\begin{equation*}
\varphi '\left(y\right){\leq}\varphi'\left(\mu\right)+\frac{\varphi'\left(x \right)-\varphi'\left(\mu\right)}{x
-\mu}(y-\mu)
\end{equation*}
for all  $y\in(\mu,x)$.  Therefore we have 
\begin{equation*}
\begin{split}
\varphi \left(x\right)-\varphi \left(\mu \right)&=\int _\mu^x\varphi '\left(y\right)\mathit{dy}\\
&\le\int _\mu^x\left\{
\varphi'\left(\mu\right)+\frac{\varphi'\left(x \right)-\varphi'\left(\mu\right)}{x
-\mu}(y-\mu)
\right\}
\mathit{dy}\\
&=\frac{\varphi '\left(x\right)+\varphi '\left(\mu \right)} 2\left(x-\mu
\right).
\end{split}
\end{equation*}
and the result follows.  
\end{proof}

Lemma 1 makes Theorem 1 easy to use as the follow results hold:
\[
\begin{split}
&\begin{cases}
\begin{split}
&\inf h(x;\mu)= \lim_{x\rightarrow a}h(x;\mu)\\
&\sup h(x;\mu)= \lim_{x\rightarrow b}h(x;\mu)
\end{split}, &\text{when } \varphi'(x) \text{ is convex}\\
\end{cases}\\
&\begin{cases}
\begin{split}
&\inf h(x;\mu)= \lim_{x\rightarrow b}h(x;\mu)\\
&\sup h(x;\mu)= \lim_{x\rightarrow a}h(x;\mu)
\end{split}, &\text{when } \varphi'(x) \text{ is concave.}
\end{cases}
\end{split}
\]
Note the limits of $h(x;\mu)$ can be either finite or infinite. The proof of Lemma 1 borrows ideas from \cite{Bennish:2003aa}.  Examples of functions $\varphi(x)$ for which $\varphi'$ is convex include $\varphi(x)=\exp(x)$ and $\varphi(x)=x^p$ for $p\ge2$ or $p\in(0,1]$.  
  Examples of functions $\varphi(x)$ for which $\varphi'$ is concave include $\varphi(x)=-\log x$ and $\varphi(x)=x^p$ for $p<0$ or $p\in[1,2]$.

\section{Examples}
\begin{example}[Moment Generating Function] For any random variable $X$ supported on $(a,b)$ with a finite variance, we can bound the moment generating function $E[e^{tX}]$ using Theorem 1 to get 
\[
\inf_{x{\in}(a,b)}\{h(x;\mu)\} \text{var}(X)\le E[e^{tX}]-e^{tE[X]}\le\sup_{x{\in}(a,b)}\{h(x;\mu)\} \text{var}(X),
\]
where 
\[
h(x;\mu)=\frac{e^{tx}-e^{t\mu}}{(x-\mu)^2}-\frac{te^{t\mu}}{x-\mu}.
\]
For $t>0$ and $(a,b)=(-\infty,\infty)$, we have 
\[
\inf h(x;\mu)=\lim_{x\rightarrow-\infty}h(x;\mu)=0\quad \text{and}\quad \sup h(x;\mu)=\lim_{x\rightarrow\infty}h(x;\mu)=\infty.
\]
So Theorem 1 provides no improvement over Jensen's inequality.  However, on a finite domain such as a non-negative random variable with $(a,b)=(0,\infty)$, a significant improvement in the lower bound is possible because
\[
\inf h(x;\mu)=h(0;\mu)=\frac{1-e^{t\mu}+t\mu e^{t\mu}}{\mu^2}>0.
\]
Similar results hold for $t<0$.  We apply this to an example from \cite{Walker:2014aa}, where $X$ is an exponential random variable with mean 1 and $\varphi(x)=e^{tx}$ with $t=1/2$.   Here the actual Jensen's gap is $\E[e^{tX}]-e^{t\E[X]}= 2-\sqrt{e}\approx.351$. 
Since $\text{var}(X)=1$, we have  
\[ 
.176\approx h(0;\mu)\le \E[e^{tX}]-e^{t\E[X]}\le \lim_{x\rightarrow\infty }h(x;\mu)=\infty.
\]
The less sharp lower bound using $\inf \varphi''(x)/2$ is 0.125. Utilizing elaborate approximations and numerical optimizations \cite{Walker:2014aa} yielded a more accurate lower bound of 0.271.  

\end{example}

\begin{example}[Arithmetic vs Geometric Mean] Let $X$ be a positive random variable on interval $(a,b)$ with mean $\mu$. Note that $-\log(x)$ is convex whose derivative is concave.  Applying Theorem 1 and Lemma 1 leads to 
\[
 \lim_{x\rightarrow b}h(x;\mu) \text{var}(X) \le -E\{\log(X)\} + \log\mu\le \lim_{x\rightarrow a}h(x;\mu)\text{var}(X),
\]
where
\[
h(x;\mu)=\frac{-\log x+\log \mu}{(x-\mu)^2}+\frac{1}{\mu(x-\mu)}.
\]

Now consider a sample of $n$ positive data points $x_1,\ldots,x_n$.  Let $\bar{x}$ be the arithmetic mean and $\bar{x}_g=(x_1x_2\cdots x_n)^{\frac{1}{n}}$ be the geometric mean.  Applying Corollary \ref{cor} gives
\[
\exp\{S^2 h(b;\bar{x})\} \le \frac{\bar{x}}{\bar{x}_g}\le \exp\{S^2 h(a;\bar{x})\},
\]
where $a$, $b$, $S^2$ are as defined in Corollary \ref{cor}.  To give some numerical results, we generated 100 random numbers from uniform distribution on [10,100].  For these 100 numbers, the arithmetic mean $\bar{x}$ is 54.830 and the geometric mean $\bar{x}_g$ is 47.509.  The above inequality becomes
\[
1.075\le \frac{\bar{x}}{(x_1x_2\cdots x_n)^{\frac{1}{n}}} = 1.154\le 1.331,
\]
which are fairly tight bounds. Replacing $h(x_n;\bar{x})$ by $\varphi''(x_n)/2$ and $h(x_1;\bar{x})$ by $\varphi''(x_1)/2$ leads to a less accurate lower bound 1.0339 and upper bound 21.698.  
\end{example}

\begin{example}[Power Mean] Let $X$ be a positive random variable on a positive interval $(a, b)$ with mean $\mu$. For any real number $s\not=0$, define the power mean as 
\[
M_s(X)=\left(E X^s\right)^{1/s}
\]
Jensen's inequality establishes that $M_s(X)$ is an increasing function of $s$. We now give a sharper inequality by applying Theorem 1. Let $r\not=0$, $Y=x^r$, $\mu_y=EY$, $p=s/r$ and $\varphi(y)=y^p$.    Note that $E X^s = E\{\varphi(Y)\}$.  Applying Theorem 1 leads to 
\[
\inf h(y;\mu_y)\text{var}(Y)\le E[X^s] - (E X^r)^p \le \sup h(y;\mu_y) \text{var}(Y),
\]
where
\[
h(y;\mu_y)=\frac{y^p-\mu_y^p}{(y-\mu_y)^2}-\frac{p \mu_y^{p-1}}{y-\mu_y}.
\]
To apply Lemma 1, note that $\varphi'(y)$ is convex for $p\ge 2$ or $p\in(0,1]$ and is concave for $p<0$ or $p\in[1,2]$ as noted in Section 2.  

Applying the above result to the case of $r=1$ and $s=-1$, we have $Y=X$, $p=-1$. Therefore
\[
\left((EX)^{-1} +\lim_{y\rightarrow a}h(y;\mu_y)\text{var}(X)\right)^{-1}\le \left(E X^{-1}\right)^{-1} \le \left((EX)^{-1} + \lim_{y\rightarrow b}h(y;\mu_y)\text{var}(X)\right)^{-1}.
\]

For the same sequence $x_1,\ldots, x_n$ generated in Example 2, we have $\bar{x}_{\text{harmonic}}=39.113$.  Applying Corollary \ref{cor} leads to 
\[
25.337\le \bar{x}_{\text{harmonic}}=39.113\le 48.905.
\]
Note that the upper bound 48.905 is much smaller than the arithmetic mean $\bar{x}=54.830$ by the Jensen's inequality.  Replacing $h(b;\bar{x})$ by $\varphi''(b)/2$ and $h(a;\bar{x})$ by $\varphi''(a)/2$ leads to a less accurate lower bound 0.8298 and 51.0839.

\end{example}

In a recent article published in the \emph{American Statistician}, \cite{Carvalho:2016aa} revisited Kolmogorov's formulation of generalized mean as 
\begin{equation}
\label{eq:K2}
E_\varphi(X)=\varphi^{-1}(E\left[\varphi(X)\right]),
\end{equation}
where $\varphi$ is a continuous monotone function with inverse $\varphi^{-1}$. The Example 2 corresponds to $\varphi(x)=-\log(x)$ and Example 3 corresponds to $\varphi(x)=x^s$.  We can also apply Theorem 1 to bound $\varphi^{-1}(E \varphi(X))$ for a more general function $\varphi(x)$.  

\begin{example}[Rao-Blackwell Estimator]
Rao-Blackwell theorem (Theorem 7.3.17 in Casella and Berger, 2002; Theorem 10.42 in Wasserman, 2013) is a basic result in statistical estimation. Let  $\hat\theta$ be an estimator of $\theta$,  $L(\theta, \hat\theta)$ be a loss function convex in $\hat\theta$, and $T$ a sufficient statistic. Then the Rao-Blackwell estimator, $\hat\theta^*=E[\hat\theta\mid T]$, satisifies the following inequality in risk function 
\begin{equation}
\label{eq:rao}
E[L(\theta, \hat\theta)]\ge E[L(\theta, \hat\theta^*)].
\end{equation}
We can improve this inequality by applying Theorem 1 to $\varphi(\hat\theta)=L(\theta, \hat\theta)$ with respect to the conditional distribution of $\hat\theta$ given $T$:
\[
E[L(\theta, \hat\theta)\mid T] - L(\theta, \hat\theta^*)\ge \inf_{x\in(a,b)} h(x;\hat\theta^*) \text{var}(\hat\theta\mid T),
\]  
where function $h$ is defined as in Theorem 1 for $\varphi(\hat\theta)$ and $P(\hat\theta\in(a,b)\mid T)=1$. Further taking expectations over $T$ gives
\[
E[L(\theta, \hat\theta)] - E[L(\theta, \hat\theta^*)]\ge E\left[\inf_{x\in(a,b)} h(x;\hat\theta^*)\text{var}(\hat\theta\mid T)\right].
\]
In particular for square-error loss, $L(\theta, \hat\theta)=(\hat\theta-\theta)^2$, we have
\[
E[(\theta-\hat\theta)^2] - E[(\theta-\hat\theta^*)^2] = E\left[\text{var}(\hat\theta\mid T)\right].
\]
Using the original Jensen's inequality only establishes the cruder inequality in Equation \eqref{eq:rao}.

\end{example}

\section{Improved bounds by partitioning}

As discussed in Example 1 above, Theorem 1 does not improve on Jensen's inequality if $\inf h(x;\mu) = 0$.  In such cases, we can often sharpen the bounds by partitioning the domain $(a,b)$ following an approach used in \cite{Walker:2014aa}. Let 
\begin{equation*}
a=x_0<x_1<{\cdots}<x_m=b,
\end{equation*}
$I_j=\left[x_{j-1},x_j\right)$, $\eta_j=P(X\in I_j)$, and $\mu_j=E(X\mid X\in I_j)$. It follows from the law of total expectation that 

\begin{equation*}
\label{eq:expectation}
\begin{split}
E[\varphi(X)]&=\sum_{j=1}^m \eta_j E[\varphi(X) \mid X\in I_j]\\
&=\sum_{j=1}^m \eta_j \varphi(\mu_j)+ \sum_{j=1}^m \eta_j \left(E[\varphi(X) \mid X\in I_j]-\varphi(\mu_j)\right).\\
\end{split}
\end{equation*}
Let $Y$ be a discrete random variable with distribution $P(Y=\mu_j)=\eta_j, j=1,2,\ldots,m$.
It is easy to see that $EY = EX$.  It follows by Theorem 1 that
\begin{equation*}
\sum_{j=1}^m\eta_j\varphi(\mu_j) = E\left[\varphi \left(Y\right)\right]\ge \varphi(EY)+\inf_{y\in[\mu_1,\mu_m]} h(y;\mu_y)\text{var}(Y).
\end{equation*}
We can also apply Theorem 1 to each $E[\varphi(X \mid X\in I_j)]-\varphi(\mu_j)$ term: 
\begin{equation*}
\label{eq:lower2}
E[\varphi(X \mid X\in I_j)]-\varphi(\mu_j)\ge \inf_{x\in I_j} h(x;\mu_j)\text{var}(X\mid X\in I_j).
\end{equation*}
Combining the above two equations, we have
\begin{equation}
\label{eq:lower3}
E[\varphi(X)]- \varphi(EX) \ge  \inf_{y\in[\mu_1,\mu_m]} h(y;\mu_y)\text{var}(Y) + \sum_{j=1}^m \eta_j \inf_{x\in I_j} h(x;\mu_j)\text{var}(X\mid X\in I_j).
\end{equation}
Replacing $\inf$ by $\sup$ in the righthand side gives the upper bound.

The Jensen gap on the left side of \eqref{eq:lower3} is positive if any of the $m+1$ terms on the right is positive. In particular, the Jensen gap is positive if there exists an interval $I\subset(a,b)$ that satisfies $\inf_{x\in I}\varphi''(x)>0$, $P(X\in I)>0$ and $\text{var}(X\mid X\in I)>0$.  Note that a finer partition does not necessarily lead to a sharper lower bound in \eqref{eq:lower3}. The focus of the partition should therefore be on isolating the part of interval $(a,b)$ in which $\varphi''(x)$ is close to 0. 

Consider example $X{\sim}N\left(\mu ,\sigma ^2\right)$ with $\mu=0$ and $\sigma=1$ and $\varphi(x)=e^x$. We divide $(-\infty,\infty)$ into three intervals with equal probabilities. This gives
\begin{table}[H]
\centering
\begin{tabular}{|c|c|c|c|c|c|}
\hline
$I_j$ & $\eta_j$ & $E[X\mid X\in I_j]$ & $\text{var}(X\mid X\in I_j)$ & $\inf_{x\in I_j} h(x;\mu_j)$ & $\sup_{x\in I_j} h(x;\mu_j)$\\
\hline
$(-\infty,-.431)$ & 1/3 & -1.091 & 0.280 & 0.000 & 0.212\\ 
\hline
$(-0.431,0.431)$ &1/3 & 0.000 & 0.060 & 0.435 & 0.580\\
\hline
$(0.431,\infty)$& 1/3 & 1.091 & 0.280 & 1.209 & $\infty$\\
\hline
\end{tabular}
\end{table}
\noindent The actual Jensen gap is $e^{\mu+\frac{\sigma}{2}}-e^{\mu}=0.649$.  The lower bound from \eqref{eq:lower3}
 is 0.409, which is a huge improvement over Jensen's bound of 0.  The upper bound $\infty$, however, provides no improvement over Theorem 1.  

To summarize, this paper proposes a new sharpened version of the Jensen's inequality. The proposed bound is simple and insightful, is broadly applicable by imposing minimum assumptions on $\varphi(x)$, and provides fairly accurate result in spite of its simple form. It can be incorporated in any calculus-based statistical course.  

\bibliographystyle{chicago}
\bibliography{jensen}

\end{document}